\documentclass[11pt, reqno]{amsart}
\usepackage{amssymb}
\usepackage{verbatim}
\usepackage[vcentermath]{youngtab}
\usepackage{float}
\usepackage{tikz-cd}
\usepackage{mathtools}
\usepackage[margin =1.32in]{geometry}

\usepackage{times}
\usepackage[T1]{fontenc}
\usepackage{mathrsfs}
\usepackage{epsfig}
\usepackage{color}

\newtheorem{thm}{Theorem}
\newtheorem{lemma}[thm]{Lemma}
\newtheorem{conjecture}[thm]{Conjecture}
\newtheorem{cor}[thm]{Corollary}

\newtheorem{defi}[thm]{Definition}

\theoremstyle{remark}

\newtheorem*{remark}{Remark}

\theoremstyle{definition}

\newtheorem{example}[thm]{Example}

\makeatletter
\newcommand{\subalign}[1]{%
  \vcenter{%
    \Let@ \restore@math@cr \default@tag
    \baselineskip\fontdimen10 \scriptfont\tw@
    \advance\baselineskip\fontdimen12 \scriptfont\tw@
    \lineskip\thr@@\fontdimen8 \scriptfont\thr@@
    \lineskiplimit\lineskip
    \ialign{\hfil$\m@th\scriptstyle##$&$\m@th\scriptstyle{}##$\crcr
      #1\crcr
    }%
  }
}
\makeatother
\def\ZZ{\mathbb{Z}}
\def\QQ{\mathbb{Q}}

\def\QQ{\mathbb{Q}}

\def\p{\mathfrak{p}}

\begin{document}

\title{Circulant $q$-Butson Hadamard Matrices}

\author{Trevor Hyde}
\address{Dept. of Mathematics\\
University of Michigan \\
Ann Arbor, MI 48109-1043\\
}
\email{tghyde@umich.edu}

\author{Joseph Kraisler}
\address{Dept. of Mathematics\\
University of Michigan \\
Ann Arbor, MI 48109-1043\\
}
\email{jkrais@umich.edu}

\date{March 14th, 2017}

\maketitle

\begin{abstract}
If $q = p^n$ is a prime power, then a $d$-dimensional \emph{$q$-Butson Hadamard matrix} $H$ is a $d\times d$ matrix with all entries $q$th roots of unity such that $HH^* = dI_d$. We use algebraic number theory to prove a strong constraint on the dimension of a circulant $q$-Butson Hadamard matrix when $d = p^m$ and then explicitly construct a family of examples in all possible dimensions. These results relate to the long-standing circulant Hadamard matrix conjecture in combinatorics.
\end{abstract}

%Let's put the butson stuff and our result right away. Then we can talk about the Hadamard matrices.

\section{Introduction}
For a prime power $q = p^n$, a \emph{$q$-Butson Hadamard matrix} ($q$-BH) of dimension $d$ is a $d\times d$ matrix $H$ with all entries $q$th roots of unity such that
\[
    H H^* = d I_d,
\]
where $H^*$ is the conjugate transpose of $H$.  A $d\times d$ matrix $H$ is said to be \emph{circulant} if 
\[
    H_{ij} = f(i-j)
\]
for some function $f$ defined modulo $d$. In this paper we investigate circulant $q$-BH matrices of dimension $d = p^m$.

\begin{thm}
\label{main1}
If $q = p^n$ is a prime power and $d = p^m$, then there exists a $d\times d$  circulant $q$-Butson Hadamard matrix if and only if $m \leq 2n$, with one exception when $(m,n,p) = (1,1,2)$.
\end{thm}

\noindent Our analysis of circulant $q$-BH matrices led us to the useful notion of \emph{fibrous functions}.
\begin{defi}
Let $d\geq 0$ and $q = p^n$ be a prime power,
\begin{enumerate}
    \item If $X$ is a finite set, we say a function $g: X \rightarrow \ZZ/(q)$ is \textbf{fibrous} if the cardinality of the fiber $|g^{-1}(b)|$ depends only on $b \bmod p^{n-1}$.
    \item We say a function $f: \ZZ/(d) \rightarrow \ZZ/(q)$ is \textbf{$\delta$-fibrous} if for each $k\not\equiv 0 \bmod d$ the function $\delta_k(x) = f(x + k) - f(x)$ is fibrous.
\end{enumerate}
\end{defi}

\noindent When $q = d = p$ are both prime, $\delta$-fibrous functions coincide with the concept of \emph{planar functions}, which arise in the study of finite projective planes \cite{do} and have applications in cryptography \cite{nkk}. Circulant $q$-BH matrices of dimension $d$ are equivalent to $\delta$-fibrous functions $f: \ZZ/(d) \rightarrow \ZZ/(q)$.

\begin{thm}
\label{equiv_intro}
Let $q = p^n$ be a prime power and $\zeta$ a primitive $q$th root of unity. There is a correspondence between $d\times d$ circulant $q$-Butson Hadamard matrices $H$ and $\delta$-fibrous functions $f:\ZZ/(d)\to\ZZ/(q)$ given by
\[
    \big(H_{ij}\big) = \big(\zeta^{f(i-j)}\big).
\]
\end{thm}

\noindent We restate our main result in the language of $\delta$-fibrous functions.
\begin{cor}
If $q = p^n$ is a prime power, then there exist $\delta$-fibrous functions $f: \ZZ/(p^m) \rightarrow \ZZ/(q)$ if and only if $m \leq 2n$ with one exception when $(m,n,p) = (1,1,2)$.
\end{cor}

\noindent It would be interesting to know if $\delta$-fibrous functions have applications to finite geometry or cryptography.

When $q = 2$, $q$-BH matrices are called \emph{Hadamard matrices}. Hadamard matrices are usually defined as $d\times d$ matrices $H$ with all entries $\pm 1$ such that $HH^t = dI_d$. Buston Hadamard matrices were introduced in \cite{butson} as a generalization of Hadamard matrices. 

Circulant Hadamard matrices arise in the theory of difference sets, combinatorial designs, and synthetic geometry \cite[Chap. 9]{ryser}. There are arithmetic constraints on the possible dimension of a circulant Hadamard matrix. When the dimension $d = 2^m$ is a power of two we have:
\begin{thm}[Turyn \cite{turyn}]
If $d = 2^m$ is the dimension of a circulant Hadamard matrix, then $m = 0$ or $m =2$. 
\end{thm}
\noindent Turyn's proof uses algebraic number theory, more specifically the fact that 2 is totally ramified in the $2^n$th cyclotomic extension $\QQ(\zeta_{2^n})/\QQ$; an elementary exposition may be found in Stanley \cite{stanley}. Conjecturally this accounts for all circulant Hadamard matrices \cite{ryser}.
\begin{conjecture}
There are no $d$-dimensional circulant Hadamard matrices for $d > 4$. 
\end{conjecture}

Circulant $q$-Butson Hadamard matrices provide a natural context within which to consider circulant Hadamard matrices. A better understanding of the former could lead to new insights on the latter. For example, our proof of Theorem \ref{main1} shows that the two possible dimensions for a circulant Hadamard matrix given by Turyn's theorem belong to larger family of circulant $q$-BH matrices with the omission of $d = 2$ being a degenerate exception. Circulant $q$-BH matrices have been studied in some specific dimensions \cite{hgs}, but overall seem poorly understood. We leave the existence of $q$-BH matrices when the dimension $d$ is not a power of $p$ for future work.

\subsection*{Acknowledgements}
The authors thank Padraig Cathain for pointers to the literature, in particular for bringing the work of de Launey \cite{launey} to our attention. We also thank Jeff Lagarias for helpful feedback on an earlier draft.

\section{Main Results}

We always let $q = p^n$ denote a prime power. First recall the definitions of $q$-Butson Hadamard and circulant matrices.

\begin{defi}
A $d$-dimensional \textbf{$q$-Butson Hadamard matrix} $H$ ($q$-BH) is a $d\times d$ matrix all of whose entries are $q$th roots of unity satisfying
\[
    HH^* = d I_d,
\]
where $H^*$ is the conjugate transpose of $H$.\\

\noindent A $d$-dimensional \textbf{circulant matrix} $C$ is a $d\times d$ matrix with coefficients in a ring $R$ such that
\[
    C_{ij} = f(i -j)
\]
for some function $f: \ZZ/(d) \rightarrow R$.
\end{defi}

\begin{example}
Hadamard matrices are the special case of $q$-BH matrices with $q = 2$. The $q$-Fourier matrix $(\zeta^{ij})$ where $\zeta$ is a primitive $q$th root of unity is an example of a $q$-BH matrix (which may also be interpreted as the character table of the cyclic group $\ZZ/(q)$.) When $q = 3$ and $\omega$ is a primitive 3rd root of unity this is the matrix:
\[
    \begin{pmatrix*}[l]
    1 & 1 & 1\\
    1 & \omega & \omega^2\\
    1 & \omega^2 & \omega^4
    \end{pmatrix*}
\]
This example is not a \emph{circulant} $3$-BH matrix. The following is a circulant $3$-BH matrix:
\[
    \begin{pmatrix}
    1 & \omega & \omega \\
    \omega & 1 & \omega \\
    \omega & \omega & 1
    \end{pmatrix}
\]
\end{example}

The remainder of the paper is divided into two sections: first we prove constraints on the dimension $d$ of a circulant $q$-BH matrix when $d$ is a power of $p$; next we introduce the concept of $\delta$-fibrous functions and construct examples of circulant $q$-BH matrices in all possible dimensions. 

\subsection*{Constraints on dimension}

Theorem \ref{main} uses the ramification of the prime $p$ in the $q$th cyclotomic extension $\QQ(\zeta)/\QQ$ to deduce strong constraints on the dimension of a $q$-BH matrix.
\begin{thm}
\label{main}
If $q = p^n$ is a prime power and $H$ is a circulant $q$-Butson Hadamard matrix of dimension $d = p^{m+n}$, then $m \leq n$.
\end{thm}

\noindent Note that our indexing of $m$ has changed from the introduction; this choice was made to improve notation in our proof. We use this indexing for the rest of the paper.

\begin{proof}[Proof of Theorem \ref{main}]
Suppose $H = (a_{i-j})$ is a circulant $q$-Butson Hadamard matrix of dimension $d= p^{m+n}$. From $H$ being $q$-BH of dimension $d$ we have
\[
    HH^* = dI_{d},
\]
hence $\det(H) = \pm d^{d/2}$ and each eigenvalue $\alpha$ of $H$ has absolute value $|\alpha| = \sqrt{d} = p^{(m+n)/2}$. On the other hand, since $H = (a_{i-j})$ is circulant, it has eigenvalues
\begin{equation}
\label{ev}
    \alpha_k = \sum_{j < d}a_j \zeta^{jk}.
\end{equation}
for $\zeta$ a primitive $d$th root of unity with corresponding eigenvector
\[
    u_k^t = \big(1,\zeta^k, \zeta^{2k},\ldots, \zeta^{(d-1)k}\big).
\]
These observations combine to give two ways of computing $\det(H)$.
\begin{equation}
\label{eq1}
    \prod_k \alpha_k = \det(H) = \pm p^{(m+n)d/2}.
\end{equation}
The identity \eqref{eq1} is the essential interaction between the circulant and $q$-BH conditions on $H$.
The prime $p$ is totally ramified in $\QQ(\zeta)$, hence there is a unique prime ideal $\p\subseteq \ZZ[\zeta]$ over $(p)\subseteq \ZZ$. Since all $\alpha_k \in \ZZ[\zeta]$, it follows from \eqref{eq1} that $(\alpha_k) = \p^{v_k}$ as ideals of $\ZZ[\zeta]$ for some $v_k \geq 0$ and for each $k$. So either $\alpha_0/\alpha_1$ or $\alpha_1/\alpha_0$ is an element of $\ZZ[\zeta]$. Say $\alpha_0/\alpha_1$ is the integral quotient. We noted $|\alpha_k| = p^{(m+n)/2}$ for each $k$, hence $|\alpha_0/\alpha_1| = 1$. The only integral elements of $\ZZ[\zeta]$ with absolute value 1 are roots of unity, hence $\alpha_0/\alpha_1 = \pm\zeta^r$ for some $r\geq 0$, hence
\begin{equation}
\label{eq2}
    \alpha_0 = \pm \zeta^r \alpha_1.
\end{equation}
\noindent By \eqref{ev} we have
\[
    \alpha_0 = \sum_{j < d} a_j \hspace{3em}
    \alpha_1 = \sum_{j < d} a_j\zeta^j.
\]
Each $j < d = p^{m+n}$ has a unique expression as $j = j_0 + j_1p^m$ where $j_0 < p^m$ and $j_1 < p^n$. Let $\omega = \zeta^{p^m}$ be a primitive $q$th root of unity. Then
\[
    \zeta^j = \zeta^{j_0 + j_1p^m} = \omega^{j_1}\zeta^{j_0}.
\]
Writing $\alpha_1$ in the linear basis $\big\{1, \zeta,\zeta^2,\ldots,\zeta^{p^m-1}\big\}$ of $\QQ(\zeta)/\QQ(\omega)$ we have
\[
    \alpha_1 = \sum_{j < d} a_j \zeta^j = \sum_{j' < p^m} b_{j'}\zeta^{j'},
\]
where $b_{j'}$ is a sum of $p^n$ complex numbers each with absolute value 1. Now \eqref{eq2} says $\alpha_0 = \pm \omega^{j_0}\zeta^{-j_1} \alpha_1$ for some $j_0$ and $j_1$, thus
\[
    \sum_{j < d}a_j = \alpha_0 = \pm\omega^{j_0}\zeta^{-j_1}\alpha_1 = \sum_{j' < p^m} \pm\omega^{j_0}b_{j'} \zeta^{j' - j_1}.
\]
Comparing coefficients we conclude that
\[
    \alpha_0 = \pm\omega^{j_0}b_{j_1},
\]
which is to say that $\alpha_0$ is the sum of $p^n$ complex numbers each with absolute value 1, hence $|\alpha_0| \leq p^n$. On the other hand we have $|\alpha_0| = p^{(m+n)/2}$. Thus $m + n \leq 2n \Longrightarrow 0 \leq m \leq n$ as desired.
\end{proof}

\begin{remark}
The main impediment to extending this result from $q = p^n$ to a general integer $q$ is that we no longer have the total ramification of the primes dividing the determinant of $H$. It may be possible to get some constraint in certain cases from a closer analysis of the eigenvalues $\alpha_k$ and ramification, but we do not pursue this.
\end{remark}

\subsection*{$\delta$-Fibrous functions and construction of circulant $q$-BH matrices}

Recall the notion of \emph{fibrous functions} from the introduction:
\begin{defi}
Let $d\geq 0$ and $q = p^n$ be a prime power,
\begin{enumerate}
    \item If $X$ is a finite set, we say $g: X \rightarrow \ZZ/(q)$ is \textbf{fibrous} if the cardinality of the fibers $|g^{-1}(b)|$ depends only on $b \bmod p^{n-1}$.
    \item We say a function $f: \ZZ/(d) \rightarrow \ZZ/(q)$ is \textbf{$\delta$-fibrous} if for each $k\not\equiv 0 \bmod d$ the function $x \mapsto f(x + k) - f(x)$ is fibrous.
\end{enumerate}
\end{defi}

\noindent Lemma \ref{interpret} is a combinatorial reinterpretation of the cyclotomic polynomials $\Phi_{p^n}(x)$. 

\begin{lemma}
\label{interpret}
Let $p$ be a prime and $\zeta$ a primitive $p^n$th root of unity.
\begin{enumerate}
    \item If $\sum_{k < p^n} b_k \zeta^k = 0$ with $b_k \in \QQ$, then $b_k$ depends only on $k \bmod p^{n-1}$.
    \item If $X$ is a finite set and $g: X \rightarrow \ZZ/(p^n)$ is a function, then $g$ is fibrous iff
    \[
        \sum_{x \in X} \zeta^{g(x)} = 0.
    \]
\end{enumerate} 
\end{lemma}

\begin{proof}
(1) Suppose $\sum_{k < p^n} b_k \zeta^k = 0$ for some $b_k\in \QQ$. Then $r(x) = \sum_{k < p^n} b_k x^k \in \QQ[x]$ is a polynomial with degree $< p^n$ such that $r(\zeta) = 0$. So there is some $s(x) \in \QQ[x]$ such that $r(x) = s(x) \Phi_{p^n}(x)$ where
\[
    \Phi_{p^n}(x) = \sum_{j < p} x^{jp^{n-1}}
\]
is the $p^n$th cyclotomic polynomial---the minimal polynomial of $\zeta$ over $\QQ$. Since $\deg \Phi_{p^n}(x) = p^n - p^{n-1}$, it follows that $\deg s(x) < p^{n-1}$. Let
\[
    s(x) = \sum_{i < p^{n-1}} a_i x^i
\]
for some $a_i \in\QQ$. Expanding $s(x)\Phi_{p^n}(x)$ we have
\[
    r(x) = s(x)\Phi_{p^n}(x) = \sum_{\subalign{i &< p^{n-1}\\ j &< p}} a_i x^{i + jp^{n-1}}.
\]
Comparing coefficients yields
\[
    b_k = b_{i + jp^{n-1}} = a_i,
\]
which is to say, $b_k$ depends only $i \equiv k \bmod p^{n-1}$.\\

\noindent (2) Suppose $g$ is fibrous. For each $i < p^{n-1}$, let $a_i = |g^{-1}(i)|$. Then
\[
    \sum_{x\in X} \zeta^{g(x)} =  \sum_{i < p^{n-1}}\sum_{j < p}a_i\zeta^{i + jp^{n-1}} = \Phi_{p^n}(\zeta)\sum_{i < p^{n-1}} a_i \zeta^i= 0.
\]
Conversely, for each $k < p^n$ let $c_k = |g^{-1}(k)|$. Then
\[
    0 = \sum_{x\in X} \zeta^{g(x)} = \sum_{k < p^n} c_k \zeta^k,
\]
and (1) implies $c_k$ depends only on $k \bmod p^{n-1}$. Hence $g$ is fibrous.
\end{proof}

Theorem \ref{equiv} establishes the equivalence between $q$-BH matrices of dimension $d$ and $\delta$-fibrous functions $f: \ZZ/(d) \rightarrow \ZZ/(q)$.

\begin{thm}
\label{equiv}
Let $q = p^n$ be a prime power. There is a correspondence between circulant $q$-Butson Hadamard matrices $H$ of dimension $d$ and $\delta$-fibrous functions $f:\ZZ/(d)\to\ZZ/(q)$ given by
\[
    \big(H_{i,j}\big) = \big(\zeta^{f(i-j)}\big).
\]
\end{thm}

\begin{proof}
Suppose $f$ is $\delta$-fibrous. Define the matrix $H = (H_{ij})$ by $H_{ij}=\zeta^{f(i-j)}$ where $\zeta$ is a primitive $q$th root of unity. $H$ is plainly circulant and has all entries $q$th roots of unity. It remains to show that $HH^* = dI_d$, which is to say that the inner product $r_{j,j+k}$ of column $j$ and column $j+k$ is 0 for each $j$ and each $k\not\equiv 0 \bmod d$. For each $k \not\equiv 0 \bmod d$ the function $\delta_k(x) = f(x + k) - f(x)$ is fibrous. Then we compute
\[
    r_{j,j+k} = \sum_{i < d}\zeta^{f(i-j + k) - f(i-j)} = \sum_{i < d}\zeta^{\delta_k(i-j)} = 0,
\]
where the last equality follows from Lemma \ref{interpret} (2).

Conversely, suppose $H = (H_{i,j})$ is a circulant $q$-Butson Hadammard matrix. Then $H_{i,j} = \zeta^{f(i-j)}$ for some function $f: \ZZ/(d)\rightarrow \ZZ/(q)$. Since $HH^* = dI_d$ we have for each $k \not\equiv 0 \bmod d$,

\[
    0 = r_{j,j+k} = \sum_{i<d} \zeta^{f(i-j + k)-f(i -j)}.
\]
Lemma \ref{interpret} (2) then implies $\delta_k(x) = f(x + k) - f(x)$ is fibrous. Therefore $f$ is $\delta$-fibrous.
\end{proof}

Lemma \ref{affine fibrous} checks that affine functions are fibrous. We use this in our proof of Theorem \ref{construction}.

\begin{lemma}
\label{affine fibrous}
If $q = p^n$ is a prime power, then for all $a\not\equiv 0\bmod q$ and arbitrary $b$, the function $f(x) = ax + b$ is fibrous.
\end{lemma}

\begin{proof}
Since $a \not\equiv 0 \bmod q$,
\[
    \sum_{j < q} \zeta^{f(j)} = \sum_{j < q} \zeta^{aj + b}= \zeta^b\sum_{j < q} (\zeta^a)^j= 0.
\]
Thus, by Lemma \ref{interpret} (2) we conclude that $f$ is fibrous.
\end{proof}

\begin{thm}
\label{construction}
If $q = p^n$ is a prime power, then there exists a circulant $q$-Butson Hadamard matrix of dimension $d = p^{m+n} = p^mq$ for each $m \leq n$ unless $(m,n,p) = (0,1,2)$.
\end{thm}

\noindent Our construction in the proof of Theorem \ref{construction} misses the family $(m,n,p) = (0,n,2)$ for each $n\geq 1$. Lemma \ref{lift} records a quick observation that circumvents this issue for $n>1$, as our construction does give circulant $2^{n-1}$-BH matrices of dimension $2^n$.
\begin{lemma}
\label{lift}
If $q = p^n$ is a prime power and there exists a circulant $q$-BH matrix of dimension $d$, then there exists a circulant $p^kq$-BH matrix of dimension $d$ for all $k\geq 0$.
\end{lemma}

\begin{proof}
Every $q$th root of unity is also a $p^kq$th root of unity, hence we may view a circulant $q$-BH matrix $H$ of dimension $d$ as a circulant $p^kq$-BH for all $k\geq 0$.
\end{proof}

\begin{proof}[Proof of Theorem \ref{construction}]
Our strategy is to first construct a sequence of functions 
\[
    \delta_k : \ZZ/(p^m)\times \ZZ/(q) \rightarrow \ZZ/(q)
\]
which are fibrous for each $k < p^mq$. If $i: \ZZ/(p^m)\times \ZZ/(q) \rightarrow \ZZ/(p^mq)$ is the bijection $i(x,y) = x + p^my$, we define $f: \ZZ/(p^mq) \rightarrow \ZZ/(q)$ such that $f(z+k) - f(z) = \delta_k(x,y)$ when $z = i(x,y)$. Hence $f$ is $\delta$-fibrous and Theorem \ref{equiv} implies the existence of a corresponding circulant $q$-BH matrix of dimension $p^mq$.

Now for each $k\geq 0$ define $\delta_k$ by
\begin{equation}
\label{d_k}
    \delta_k(x,y) = ky + \sum_{j < k} S_j(x),
\end{equation} 
where
\[
    S_j(x) = \sum_{i < j} \chi(x+i),\hspace{2em} \chi(x) = \begin{cases} 1 & x \equiv -1\bmod p^m,\\ 0 & \text{otherwise.}\end{cases}
\]
Observe that $S_j(x)$ counts the integers in the interval $[x, x+j)$ congruent to $-1 \bmod p^m$ (which only depends on $x \bmod p^m$.) Any interval of length $p^mj_1$ contains precisely $j_1$ integers congruent to $-1 \bmod p^m$. For each $j < p^mq$, write $j = j_0 + p^mj_1$ with $j_0 < p^m$ and $j_1 < q$, then
\begin{equation}
\label{id1}
    S_j(x) = \sum_{i < j_0 + p^mj_1} \chi(x + i) = \sum_{i < j_0}\chi(x + i) + j_1 = S_{j_0}(x) + j_1.
\end{equation}
We show that $\delta_k$ is fibrous when $k < p^mq$. If $k \not\equiv 0 \bmod q$, then for each $x = x_0$ the function $\delta_k(x_0,y)$ is affine hence fibrous by Lemma \ref{affine fibrous}. So $\delta_k(x,y)$ is fibrous. Now suppose $k = k'q$ for some $k' < p^m$. Using \eqref{id1} we reduce \eqref{d_k} to
\begin{equation}
\label{id2}
    \delta_k(x,y) = \sum_{j < k'q} S_j(x) = \sum_{j_0 + p^mj_1 <k'q} S_{j_0}(x) + j_1 = \ell \sum_{j_0 < p^m}S_{j_0}(x) + p^m\tbinom{\ell}{2},
\end{equation}
where $k'q = p^m(k'p^{n-m}) = p^m\ell$. Here we use our assumption $m\leq n$.
The definition of $\chi$ implies
\[
    S_{j_0}(x) = \begin{cases} 0 & j_0 < p^m - x\\ 1 & j_0 \geq p^m - x, \end{cases} \hspace{1em}\Longrightarrow\hspace{1em} \sum_{j_0 < p^m} S_{j_0}(x) = x,
\]
whence $\delta_k(x,y) = \ell x + p^m \tbinom{\ell}{2}$. Since $k' < p^m$ it follows that $\ell = k'p^{n-m} < p^n = q$, so $\delta_k(x,y)$ is affine hence fibrous by Lemma \ref{affine fibrous}.

Define $f:\ZZ/(p^mq) \rightarrow \ZZ/(q)$ by $f(k) = \delta_k(0,0)$. For this to be well-defined, it suffices to check that $\delta_{k+p^mq}(x,y) = \delta_k(x,y)$ with arbitrary $k$. By \eqref{d_k},
\begin{align*}
    \delta_{k + p^mq}(x,y) &= (k + p^mq)y + \sum_{j < k + p^mq} S_j(x)\\
    &= ky + \sum_{j < k}S_j(x) + \sum_{j < p^mq}S_{j +k}(x)\\
    &= \delta_k(x,y) + \sum_{j < p^mq}S_{j+k}(x).
\end{align*}
The argument leading to \eqref{id2} gives
\[
    \sum_{j < p^mq}S_{j+k}(x) = qx + p^m\tbinom{q}{2} = p^m\tbinom{q}{2}.
\]
Finally, $p^m\binom{q}{2} \equiv 0 \bmod q$ unless $(m,n,p) = (0,n,2)$. Lemma \ref{lift} implies that constructing an example for $(1,n-1,2)$ implies the existence of example for $(0,n,2)$, hence we proceed under the assumption that either $p\neq 2$ or $p=2$ and $m>0$. The case $(m,n,p) = (0,1,2)$ is an exception as one can check explicitly that there are no 2-dimensional Hadamard matrices. Hence it follows that $f$ is well-defined. Let $i:\ZZ/(p^m)\times \ZZ/(q) \rightarrow \ZZ/(p^mq)$ be the bijection $i(x,y) = x + p^my$. To finish the construction we suppose $z = i(x,y)$, show
\[
    f\big(z+k\big) - f\big(z\big) = \delta_k(x,y),
\]
and then our proof that $\delta_k$ is fibrous for all $k < p^mq$ implies $f$ is $\delta$-fibrous. Theorem \ref{equiv} translates this into the existence of a circulant $q$-BH matrix of dimension $p^{m+n}$. Now,
\begin{align*}
    f(z+k) - f(z) &= \delta_{z+k}(0,0) - \delta_z(0,0) = \sum_{z\leq j < z + k} S_j(0) = \sum_{j - z < k} \sum_{i-z < j-z} \chi(i)\\
    &= \sum_{j' < k} \sum_{i' < j'} \chi(x + i') = \sum_{j' < k} S_{j'}(x) = \delta_k(x,y).
\end{align*}
\end{proof}

\begin{remark}
Padraig Cathain brought the work of de Launey \cite{launey} to our attention after reading an initial draft. There one finds a construction of circulant $q$-Butson Hadamard matrices of dimension $q^2$ for all prime powers $q$ which appears to be closely related to our construction in Theorem \ref{construction}.
\end{remark}

\begin{example}
We provide two low dimensional examples to illustrate our construction. First we have an 8 dimensional circulant $4$-BH matrix.
\[
    \begin{pmatrix*}[r]
    1 & -1 & i & 1 & 1 & 1 & i & -1 \\
    -1 & 1 & -1 & i & 1 & 1 & 1 & i \\
    i & -1 & 1 & -1 & i & 1 & 1 & 1 \\
    1 & i & -1 & 1 & -1 & i & 1 & 1 \\
    1 & 1 & i & -1 & 1 & -1 & i & 1 \\
    1 & 1 & 1 & i & -1 & 1 & -1 & i \\
    i & 1 & 1 & 1 & i & -1 & 1 & -1 \\
    -1 & i & 1 & 1 & 1 & i & -1 & 1 \\
    \end{pmatrix*}
\]

\noindent Let $\omega$ be a primitive 3rd root of unity. The following is a 9 dimensional circulant $3$-BH matrix.
\[
    \begin{pmatrix*}[l]
    1 & \omega^2 & \omega & 1 & 1 & 1 & 1 & \omega & \omega^2 \\
    \omega^2 & 1 & \omega^2 & \omega & 1 & 1 & 1 & 1 & \omega \\
    \omega & \omega^2 & 1 & \omega^2 & \omega & 1 & 1 & 1 & 1 \\
    1 & \omega & \omega^2 & 1 & \omega^2 & \omega & 1 & 1 & 1 \\
    1 & 1 & \omega & \omega^2 & 1 & \omega^2 & \omega & 1 & 1 \\
    1 & 1 & 1 & \omega & \omega^2 & 1 & \omega^2 & \omega & 1 \\
    1 & 1 & 1 & 1 & \omega & \omega^2 & 1 & \omega^2 & \omega \\
    \omega & 1 & 1 & 1 & 1 & \omega & \omega^2 & 1 & \omega^2 \\
    \omega^2 & \omega & 1 & 1 & 1 & 1 & \omega & \omega^2 & 1 \\
    \end{pmatrix*}
\]
\end{example}

Corollary \ref{cor_fib} is an immediate consequence of our main results by Theorem \ref{equiv}.
\begin{cor}
\label{cor_fib}
If $q = p^n$ is a prime power and $d = p^{m+n}$, then there exists a $\delta$-fibrous function $f: \ZZ/(p^{m+n}) \rightarrow \ZZ/(q)$ iff $m \leq n$, with the one exception of $(m,n,p) = (0,1,2)$.
\end{cor}

\subsection*{Closing remarks}
Our analysis focused entirely on the existence of circulant $p^n$-BH matrices with dimension $d$ a power of $p$. The number theoretic method of Theorem \ref{main} cannot be immediately adapted to the case where $d$ is not a power of $p$, although as we noted earlier, it may be possible to get some constraint with a closer analysis of the eigenvalues of a circulant matrix and the ramification over the primes dividing $d$ in the $d$th cyclotomic extension $\QQ(\zeta)/\QQ$.

The family of examples constructed in Theorem \ref{construction} was found empirically. It would be interesting to know if the construction extends to any dimensions which are not powers of $p$.

\end{document}